\documentclass[a4paper,10pt,reqno, english]{amsart}  
\usepackage[utf8]{inputenc}
\usepackage[T1]{fontenc}
\usepackage{amsmath,amsthm}
\usepackage{amsfonts,amssymb,enumerate}
\usepackage{url,paralist}
\usepackage{mathtools}  
\usepackage[colorlinks=true,urlcolor=blue,linkcolor=red,citecolor=magenta]{hyperref}
\usepackage{enumerate}
\usepackage{anysize}
\usepackage{tikz}

\theoremstyle{plain}
\newtheorem{thm}{Theorem}[section]

\newtheorem{conj}[thm]{Conjecture}

\newtheorem*{thm*}{Theorem}

\theoremstyle{definition}

\newtheorem{ex}[thm]{Example}
\newtheorem{rem}[thm]{Remark}

\newcommand{\R}{\mathbb{R}}

\DeclareMathOperator{\conv}{\mathrm{conv}}
\DeclareMathOperator{\aff}{\mathrm{aff}}

\begin{document}
% -------------------------------------------------------------------%

\title[On Reay's conjecture and Conway's thrackle conjecture]{On Reay's relaxed Tverberg conjecture and generalizations of Conway's thrackle conjecture}

% -------------------------------------------------------------------%

% -------------------------------------------------------------------%

\author[M. Asada \and R. Chen \and F. Frick \and F. Huang \and M. Polevy \and D. Stoner \and L.H. Tsang \and Z. Wellner]{Megumi Asada \and Ryan Chen \and Florian Frick \and Frederick Huang \and Maxwell Polevy \and David Stoner \and Ling Hei Tsang \and Zoe Wellner}

\address[MA]{Department of Mathematics and Statistics, Williams College, Williamstown, MA 01267, USA}
\email{maa2@williams.edu}

\address[RC]{Department of Mathematics, Princeton NJ 08544, USA}
\email{rcchen@princeton.edu}

\address[FF, FH, ZW]{Department of Mathematics, Cornell University, Ithaca, NY 14853, USA}
\email{\{ff238, fh243, zaw5\}@cornell.edu}

\address[MP]{Department of Mathematics, Northeastern University, Boston, MA 02115, USA}
\email{polevy.m@husky.neu.edu}

\address[DS]{Mathematics Department, Harvard University, 1 Oxford St, Cambridge, MA 02138, USA}
\email{dstoner@college.harvard.edu}

\address[LHT]{Department of Mathematics, The Chinese University of Hong Kong}
\email{henryt918@gmail.com}

\date{\today}
\maketitle

% -------------------------------------------------------------------%

\begin{abstract}
\small Reay's relaxed Tverberg conjecture and Conway's thrackle conjecture are open problems
about the geometry of pairwise intersections. Reay asked for the minimum number of points in
Euclidean $d$-space that guarantees any such point set admits a partition into $r$ parts, any $k$
of whose convex hulls intersect. Here we give new and improved lower bounds for
this number, which Reay conjectured to be independent of~$k$. We prove a colored version of 
Reay's conjecture for $k$ sufficiently large, but nevertheless $k$ independent of dimension~$d$. 
Requiring convex hulls to intersect pairwise severely restricts combinatorics. 
This is a higher-dimensional analog of Conway's thrackle conjecture or its linear special case.
We thus study convex-geometric and higher-dimensional analogs of the thrackle conjecture alongside
Reay's problem and conjecture (and prove in two special cases) that the number of convex sets 
in the plane is bounded by the total number of vertices they involve whenever there exists a 
transversal set for their pairwise intersections. We thus isolate a geometric property that leads to
bounds as in the thrackle conjecture. We also establish tight bounds for the number of facets
of higher-dimensional analogs of linear thrackles and conjecture their continuous generalizations.
\end{abstract}

\section{Introduction}

\noindent
Given a finite point set in~$\R^d$ the intersection pattern of convex hulls determined by subsets of those points is the focus of \emph{Tverberg-type theory}.
The namesake of the area, Helge Tverberg, established in 1966 that for any ${(r-1)(d+1)+1}$ points in~$\R^d$ there exists a partition 
into $r$ parts $X_1, \dots, X_r$ such that $\conv X_1 \cap \dots \cap \conv X_r \ne \emptyset$, and this number of points is optimal in general~\cite{tverberg1966}. 
Since then a multitude of extensions and variants of this result have been proven; see for instance the recent survey article~\cite{barany2016}.

Many seemingly simple questions of Tverberg-type remain open --- among them a conjecture of Reay~\cite{reay1979}: for any $r \ge 2$ and $d \ge 1$ 
there are $(r-1)(d+1)$ points in~$\R^d$ such that for any partition of them into $r$ parts, two of them have disjoint convex hulls. This
would imply that there is no relaxation of Tverberg's theorem, where fewer than $(r-1)(d+1)+1$ points can be partitioned into $r$ sets
of pairwise intersecting convex hulls. More generally, this problem has been studied for $k$-fold intersections among the $r$ convex hulls
instead of only pairwise intersections. This was done already by Reay and later by Perles and Sigron~\cite{perles2007}.

Reay's problem seeks to understand the pairwise intersection pattern of disjoint faces in a simplicial complex~$K$ when affinely mapped to Euclidean space. Conversely,
if we are given that all facets have nonempty pairwise intersections, how does this restrict the possible combinatorics of~$K$? In the special
case of graphs this would be answered by Conway's thrackle conjecture: a \emph{thrackle} is a graph that can be drawn in the plane in such a 
way that any pair of edges intersects precisely once, either at a common vertex or a transverse intersection point. Conway conjectured 
that in any thrackle the number of edges is at most the number of vertices. This has remained open but is simple to prove if all edges 
are required to be straight line segments, that is convex; see Erd\H os~\cite{erdos1946}. 
It is an open question whether one needs to distinguish between the affine and continuous theory for thrackles;
this distinction is significant for Tverberg-type results~\cite{blagojevic2015-2, frick2015, mabillard2015}.
Not wanting to restrict our attention to $1$-dimensional objects, we set out to find convex-geometric
and higher-dimensional analogs of Conway's thrackle conjecture as a true counterpart of Reay's problem. 

\textbf{Our contributions.} Denote by $T(d,r,k)$ the minimum number~$n$ such that any~$n$ points $a_1,\ldots,a_n$ in~$\R^d$ (not necessarily distinct) admit a partition of the indices $\{1,\ldots,n\}$ into~$r$ pairwise disjoint sets $I_1, \dots, I_r$ such that any size~$k$ subfamily of
$\{\conv(a_i)_{i\in I_1},\dots,\conv(a_r)_{i\in I_r}\}$ has nonempty intersection. In Section~\ref{sec:lower-bounds} we give new
and improved lower bounds for the numbers~$T(d,r,k)$. We show that $T(d+1,r,k) \ge T(d,r,k) + k-1$, see
Theorem~\ref{thm:reay1}, and $T(d,r,k)\ge r(\frac{k-1}{k}\cdot d+1)$, see Theorem~\ref{thm:reay2}. 

Perles and Sigron~\cite{perles2007} showed that $T(d,r,k) = (r-1)(d+1)+1$ for specific values of~$k$; see Theorem~\ref{thm:reay-known}
for details. However, in those cases $k$ grows linearly with the dimension~$d$, and in fact Perles and Sigron do not
believe that $T(r,d,k) = (r-1)(d+1)+1$ in general. In contrast, Theorem~\ref{thm:reay-bl} establishes a colorful analog
of Reay's conjecture for any dimension~$d$ and a constant~$k$. 

Given a collection $C_1, \dots, C_m \subseteq \R^2$ of convex polygons on a total number of $n$ vertices such that any 
two polygons have nonempty intersection, it is simple to see that the naive extension $m \le n$ of the linear case of the thrackle conjecture 
cannot hold in general. Here we isolate a feature of the pairwise intersection pattern of convex sets that allows us to prove 
an extension of the linear thrackle conjecture: we establish the bound $m \le n$ if the full-dimensional $C_i$ are vertex-disjoint from one another
and there is a \emph{transversal set}~$W$ that contains all vertices and possibly more points such that $|C_i \cap C_j \cap W| = 1$ for all $i \ne j$;
see Theorem~\ref{thm:conv-thrackle}. We further conjecture that it is superfluous to require the full-dimensional $C_i$ to be vertex-disjoint; 
see Conjecture~\ref{conj:EdgesBoundedByVertices}. It is a purely combinatorial statement about pairwise intersection patterns
of arbitrary sets $C_1, \dots, C_m$ (not even necessarily contained in any~$\R^d$), that if there is a transversal of pairwise intersections~$W$, 
that is $|C_i\cap C_j \cap W|=1$ for all $i\ne j$, then $m \le |W|$; see Theorem~\ref{thm:comb-thrackle}.

We present higher-dimensional generalizations of the linear thrackle conjecture in Section~\ref{sec:high-dim} and
conjecture their continuous analogs.

\section*{Acknowledgements}

\noindent
This research was performed during the \emph{Summer Program for Undergraduate Research} 2016 at Cornell University. 
The authors are grateful for the excellent research conditions provided by the program. 
MA is supported by a Clare Boothe Luce scholarship. 
RC is supported by the Princeton University Mathematics Department. 
MP is supported by the Northeastern University Mathematics Department and the Northeastern University Scholars Program. 
DS is supported by a Watson--Brown Foundation scholarship. 
LHT is supported by the Mathematics Department, Science Faculty and Summer Undergraduate Research Programme 2016 of the Chinese University of Hong Kong.

\section{Lower bounds for Reay's relaxed Tverberg conjecture}
\label{sec:lower-bounds}

\noindent
Recall that $T(d,r,k)$ denotes the minimum number~$n$ such that any~$n$ points $a_1,\ldots,a_n$ in~$\R^d$ (not necessarily distinct) admit a partition of the indices $\{1,\ldots,n\}$ into~$r$ pairwise disjoint sets $I_1, \dots, I_r$ such that any size~$k$ subfamily of
$\{\conv(a_i)_{i\in I_1},\dots,\conv(a_r)_{i\in I_r}\}$ has nonempty intersection. By Tverberg's theorem $T(d,r,k) \le (r-1)(d+1)+1$,
and since that theorem is tight we have the equality $T(d,r,r) = (r-1)(d+1)+1$. Reay conjectured that in fact 
this bound is tight even for smaller~$k$, that is, $T(d,r,k) =(r-1)(d+1)+1$ for all $2 \le k \le r$. Reay's
conjecture is known to be true in some cases, and there are a few general lower bounds for the number~$T(d,r,k)$.
We collect these results here:

\begin{thm}
\label{thm:reay-known}
	We have the following lower bounds for $T(d,r,k)$:
\begin{compactenum}[(i)]
	\item Let $2 \le k \le d$, $k \le r$, and $d\ge 2$. Then $T(d,r,k) \ge (r-1)k$, $T(2,r,2) = 3r-2$, and $T(d,d+1,d)\geq (r-1)(d+1)$. Also, for $r\geq3$, we have $T(3,r,2)\geq 3r$; see Reay~\cite{reay1979}.
	\item Let $d + 1 \le 2k - 1$ or $k < r < \frac{d+1}{d+1-k}k$. Then $T(d,r,k) = (r-1)(d+1)+1$. Also $T(3,4,2)=13$ and $T(5,3,2)=13$; see Perles and Sigron~\cite{perles2007}.
	\item We have that $T(d,r,2) \ge r(\lfloor\frac{d}{2} \rfloor+1)$; see Ziegler~\cite{Bacher2011}.
\end{compactenum}
\end{thm} 

We observe that the general lower bounds for $T(d,r,k)$ that can be found in the (traditional) literature do not even depend on~$d$.
The best lower bounds for pairwise intersections $k=2$ seem to follow from Ziegler's reply to a \texttt{mathoverflow} post by Roland Bacher. 
Ziegler puts points in cyclic position. We extend his reasoning to larger $k > 2$ by putting points in strong general position; see Theorem~\ref{thm:reay2}.

It is simple to see that Tverberg's theorem is tight. For example any sufficiently generic point set will show the
tightness. Alternatively, this can also be verified by an induction on dimension; see de Longueville~\cite{delongueville2001}. 
We will use similar arguments to establish general lower bounds for~$T(d,r,k)$.

\begin{thm}
\label{thm:reay1}
	Let $d\ge2$ and $2\le k\le r$ be integers. Then $T(d+1,r,k) \ge T(d,r,k) + k-1$ and in particular 
	$T(d,r,k)\ge 3r-2+(k-1)(d-2)$.
\end{thm}

\begin{proof}
	Let $X \subseteq \R^d$ be a set of $T(d,r,k)-1$ points such that for any partition $X_1, \dots, X_r$ of~$X$
	into $r$ parts there are $k$ sets whose convex hulls avoid a common point of intersection. We will
	explicitly construct a set $Y \subseteq \R^{d+1}$ of $T(d,r,k)+k-2$ points with the same property. To this 
	end place $\R^d$ as the hyperplane~$\R^d \times\{0\}$ into~$\R^{d+1}$. Let $Y$ consist of the points in~$X$ and
	$k-1$ additional points strictly on the positive side of~$\R^d \times\{0\}$.
	
	Suppose $Y$ had a partition into $r$ sets $Y_1, \dots, Y_r$ such that for every $k$ of these sets their
	convex hulls intersect. We claim that $Y_1 \cap X, \dots, Y_r \cap X$ is a partition of~$X$ with the
	same property: for any $k$ of the~$Y_i$, say $Y_1, \dots, Y_k$, at least one $Y_j$ is entirely contained 
	in~$X$ and thus there is a point of intersection among their convex hulls in~$\R^d \times\{0\}$. But this is only possible
	if $\conv(Y_1 \cap X) \cap \dots \cap \conv(Y_k \cap X) \ne \emptyset$. Thus $Y_1 \cap X, \dots, Y_r \cap X$ is a 
	partition of~$X$ such that any $k$ of these sets have intersecting convex hulls --- a contradiction.
	
	The bound $T(d,r,k)\ge 3r-2+(k-1)(d-2)$ now follows inductively starting from $T(2,r,k) = 3r-2$ 
	given by Theorem~\ref{thm:reay-known}.
\end{proof}

\begin{rem}
	Theorem~\ref{thm:reay1} recovers the tightness of Tverberg's theorem for~$k=r$.
\end{rem}

A point set $X \subset \R^d$ is said to be in \emph{strong general position} if for any $r \ge 2$ and any disjoint subsets 
$X_1,\dots, X_r$ of $X$ the codimension of $\bigcap_i \aff(X_i)$ is equal to the sum of the codimensions
of $\aff(X_i)$ or $\bigcap_i \aff(X_i)$ is empty; see Reay~\cite{reay1979-2}, Doignon and Valette~\cite{doignon1977},
and Perles and Sigron~\cite{perles2014}.

\begin{thm}
\label{thm:reay2}
	Let $d\ge1$ and $2\le k\le r$ be integers. Then $T(d,r,k)\ge r(\frac{k-1}{k}\cdot d+1)$.
\end{thm}

\begin{proof}
	Any point set $X \subseteq \R^d$ in strong general position admitting a partition into $r$ parts $X_1, \dots, X_r$ such that any $k$ of the sets have intersecting convex hulls 
	has at least $r(\frac{k-1}{k}\cdot d+1)$ points. Denote the dimension of $\aff(X_i)$ by~$d_i$.
	Suppose $\sum_{1\le i\le r}{d_i} < r\cdot\frac{k-1}{k}d$. Then we can find indices $i_1\dots,i_k$ 
	such that $\sum_{1\le j\le k}{d_{i_j}} < (k-1)d$. But we know $\bigcap_{i_1<\dots<i_k} \conv(X_i) \ne \emptyset$ 
	which implies $\sum_{1\le j\le k}{d-d_{i_j}} \le d$, a contradiction. So we have $\sum_{1\le i\le r}{d_i} \ge r\cdot\frac{k-1}{k}d$, 
	and $|X_i|\ge d_i+1$ implies $\sum_{1\le i\le r}{|X_i|} \ge r(\frac{k-1}{k}\cdot d+1)$, as desired.
\end{proof}

\begin{rem}
	Note that $r=k$ here also recovers the tightness of Tverberg's theorem.  This bound can be 
	rewritten $(r-1)(d+1)+1-(r-k)\cdot d/k$, and for $k=r-1$ and $k\geq d+1$ the bound recovers $T(d,r,k)=(r-1)(d+1)+1$, which follows alternately from Helly's Theorem. This bound is better than the bound in Theorem~\ref{thm:reay1} for $d$ or $k$ sufficiently large.
\end{rem}

\section{Proof of a colored version of Reay's conjecture}

\noindent
Reay's conjecture is known to be true only for $k$-fold intersections, where $k$ grows linearly with~$d$. 
Here we present a variant of Reay's conjecture that turns out to be true for $k > \lceil \frac{r}{2} \rceil$ in any 
dimension~$d$. We view this as further evidence that the conjecture is true. Our variant is a $k$-fold
analog of the following conjecture which is open in general:

\begin{conj}[B\'ar\'any--Larman conjecture]
\label{conj:bl}
	Given sets $C_0, \dots, C_d \subseteq \R^d$ of cardinality~$r$, there are pairwise disjoint sets
	$X_1, \dots, X_r \subseteq \dot{\bigcup} C_i$ such that $|X_i \cap C_j| \le 1$ for every $i$ and~$j$ and
	$\conv(X_1) \cap \dots \cap \conv(X_r) \ne \emptyset$. 
\end{conj}

B\'ar\'any and Larman~\cite{barany1992} proved that this conjecture holds in the plane. Lov\'asz observed 
that the case $r=2$ is an immediate consequence of the Borsuk--Ulam theorem; this was remarked on in~\cite{barany1992}. 
More generally, the truth of this conjecture was established for $r+1$ a prime by Blagojevi\'c, Matschke, and Ziegler~\cite{blagojevic2015}. 
Here we show that in general one cannot even delete a single point and still find sets $X_1, \dots, X_r$ as in 
Conjecture~\ref{conj:bl} such that the convex hulls of any $k > \lceil \frac{r}{2} \rceil$ of them intersect.

\begin{thm}
\label{thm:reay-bl}
	Let $d \ge 1$, $r \ge 2$ and $k > \lceil \frac{r}{2} \rceil$ be integers. There are point sets $C_1, \dots, C_d 
	\subseteq \R^d$ of cardinality~$r$, and $C_0$ of cardinality~$r-1$, such that for any $r$ pairwise 
	disjoint sets $X_1, \dots, X_r \subseteq \dot{\bigcup} C_i$ with $|X_i \cap C_j| \le 1$ for every $i$ and~$j$, the convex hulls of some $k$ of them have 
	empty intersection.
\end{thm}

\begin{proof}
	We construct the point set $\dot{\bigcup} C_i$ by induction over dimension. The theorem holds for any
	set $C_0 \subseteq \R^0$ of cardinality~$r-1$, since any partition of $C_0$ into $r$ parts must 
	include the empty set. Having inductively constructed $C_0, \dots, C_d \subseteq \R^d$ as in the
	statement of the theorem, we place $\R^d$ as the hyperplane~$\R^d\times\{0\}$ in~$\R^{d+1}$ and add point set~$C_{d+1}$:
	place $\lceil \frac{r}{2} \rceil$ points of $C_{d+1}$ above $\R^d \times\{0\}$ and $\lfloor \frac{r}{2} \rfloor$ 
	points below. For any $r$ pairwise disjoint sets in $\dot{\bigcup} C_i$ any intersection among the convex 
	hulls of $k$ of them must already occur in~$\R^d \times \{0\}$ since no convex hull can contain
	two points of~$C_{d+1}$, and this finishes the induction.
\end{proof}

In particular, for $k = r$ this shows that the B\'ar\'any--Larman conjecture is tight in the sense that not even
a single point may be deleted in general.

\section{Convex generalizations of Conway's thrackle conjecture}

\noindent
Recall that a \emph{thrackle} is a graph that can be drawn in the plane such that any pair of edges intersects
precisely once, either at a common vertex or at a point of transverse intersection. Conway conjectured that
in any thrackle the number of edges does not exceed the number of vertices. This is simple to prove if all
edges are straight line segments, see Erd\H os~\cite{erdos1946} for a short proof of this \emph{linear thrackle conjecture}, 
but has remained open in general. Lov\'asz, Pach, and Szegedy~\cite{lovasz1997} proved that any 
thrackle on $n$ vertices has at most $2n-3$ edges. This bound was improved to roughly $1.428n$ by 
Fulek and Pach~\cite{fulek2010}.

Here we are interested in convex-geometric generalizations of the linear thrackle conjecture, where we replace
straight edges by more general convex sets. The naive conjecture that if $C_1, \dots, C_m$ are convex polygons
in the plane on a total number of $n$ vertices with pairwise nonempty intersections, then $m \le n$ is wrong:
consider the vertices of a regular $7$-gon and the twenty-one triangles containing precisely one edge of the 
$7$-gon.

If, however, the pairwise intersections admit a transversal set~$W$ as explained below, then we conjecture that
the number of convex sets is bounded by the total number of vertices:

\begin{conj}
\label{conj:EdgesBoundedByVertices}
	Let $W\subseteq \R^2$ be a finite set of points, $V\subseteq W$ a set of $n$ points, $C_1,\dots, C_m$ 
	distinct convex hulls of subsets of~$V$ and $|C_i\cap C_j \cap W|=1$ for all $i\ne j$. Then $m \le n$.
\end{conj}

A system of convex sets as in Conjecture~\ref{conj:EdgesBoundedByVertices} is a \emph{thrackle of convex sets}. 
If all the $C_i$ have two elements, that is, they are edges, then this reduces to the linear case of Conway's 
thrackle conjecture. Here the transversal set $W$ consists of all vertices and intersection points.
Theorem~\ref{thm:conv-thrackle} is special case of this conjecture, which is properly
stronger than the linear case of the thrackle conjecture. 

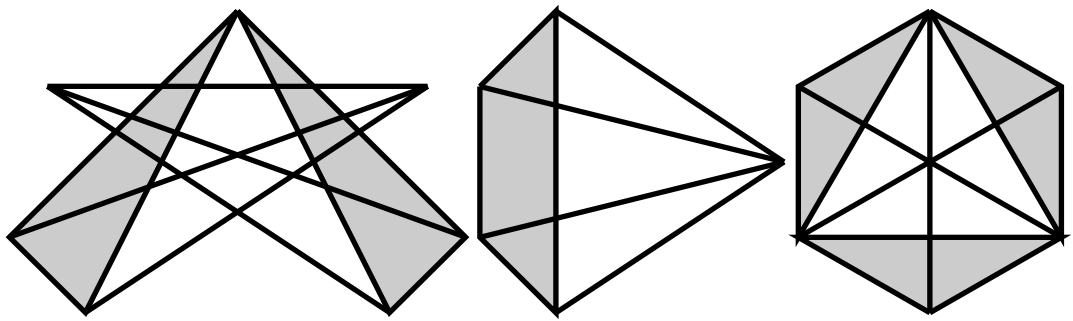
\begin{figure}[h!]
\begin{tikzpicture}
\filldraw[fill=black!20!white, draw=black, line width=2](0, 0)--(-3, -3)--(-2, -4)--(0, 0);
\filldraw[fill=black!20!white, draw=black, line width=2](0, 0)--(3, -3)--(2, -4)--(0, 0);
\path[line width=2] 
	(-2.5, -1) edge (2, -4)
    (-2.5, -1) edge (3, -3)
    (2.5, -1) edge (-2, -4)
  	(2.5, -1) edge[line width=2] (-2.5, -1)
    (2.5, -1) edge (-3, -3);

\end{tikzpicture}
\begin{tikzpicture}
\filldraw[fill=black!20!white, draw=black, line width=2](0, 1)--(0, -1)--(1, -2)--(1, 2)--(0, 1);

\path[line width=2] 
	(0, 1) edge (4, 0)
	(0, -1) edge (4, 0)
    (1, 2) edge (4, 0)
    (1, -2) edge (4, 0);

\end{tikzpicture}
\begin{tikzpicture}
\filldraw[fill=black!20!white, draw=black, line width=2](0, 2)--(1.73, 1)--(1.73, -1)--(0,2);
\filldraw[fill=black!20!white, draw=black, line width=2](0, 2)--(-1.73, 1)--(-1.73, -1)--(0,2);
\filldraw[fill=black!20!white, draw=black, line width=2](0, -2)--(1.73, -1)--(-1.73, -1)--(0,-2);
\path[line width=2] 
	(0, 2) edge (0, -2)
    (1.73, 1) edge (-1.73, -1)
    (1.73, -1) edge (-1.73, 1);

\end{tikzpicture}
\caption{Other examples of tight thrackles for Conjecture~\ref{conj:EdgesBoundedByVertices}}
\end{figure}

\begin{ex}
\label{ex:proj}
Tight examples for Conjecture~\ref{conj:EdgesBoundedByVertices}
can be obtained from finite projective planes; see Chapter 19 of van Lint and Wilson~\cite{vanlint2001} for an introduction
to combinatorial designs. A \emph{projective plane} is an incidence relation among an abstract set
of points and an abstract set of lines such that any two distinct points are incident to exactly one line, any two distinct
lines are incident to exactly one point, and there are four points such that no line is incident with three of them. In a
finite projective plane the number of points is equal to the number of lines. Finite projective planes on $q^2+q+1$ points,
with the \emph{order} $q$ a power of a prime, are simple to construct, while it is unknown whether projective planes
of order that is not a prime power exist. Given a projective plane with $n$ points and $n$ lines, consider a convex 
$n$-gon in the plane with vertices in bijection with points, and let $C_1, \dots, C_n$ be those convex sets that 
are determined by lines of the projective plane. Then any two distinct sets intersect at a common vertex and no
other vertices. Thus the set of vertices is a transversal set in the sense of Conjecture~\ref{conj:EdgesBoundedByVertices}
and the number of convex sets is equal to the total number of vertices.
\end{ex}

\begin{thm} 
\label{thm:conv-thrackle}
	Conjecture~\ref{conj:EdgesBoundedByVertices} holds in the case that the vertex sets of $C_i, C_j$ are 
	disjoint whenever $C_i, C_j$ are both 2-dimensional.
\end{thm}

\begin{proof}
Each vertex is incident to at most one $2$-dimensional set. Therefore, the neighborhood of a given vertex
consists of some rays along with at most one wedge, which represents a $2$-dimensional convex set. 
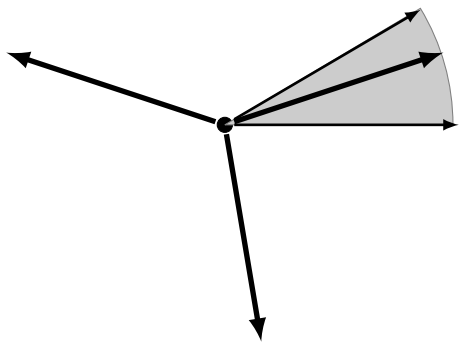
\begin{figure}[h!]
\centering
\begin{tikzpicture}[every path/.style={>=latex}]
\filldraw[fill=black] (0, 0) circle (1mm);
\filldraw[fill=black!20!white, draw=black!50!white]
    (0,0) -- (30mm,0mm) arc (0:31:30mm) -- (0,0);
\node (A) at (0, 0){};
\node (B) at (3, 1){};
\node (C) at (3.2, 0){};
\node (F) at (0.5, -3){};
\node (D) at (2.7, 1.6){};
\node (E) at (-3, 1){};
\path[->, line width=2]
	(A) edge (B)
    (A) edge (F)
    (A) edge (E);
\path[->, line width=1]
    (A) edge (C)
    (A) edge (D);
\end{tikzpicture}
\caption{Example configuration about some vertex}
\end{figure}
We describe a surjection from a subset of the vertices onto the set of convex sets. Each vertex selects at most one incident set~$C_i$:

\begin{compactitem}
	\item Case 1: If there are no wedges around~$v$, then if the measure of the clockwise angle from some ray 
		to every other ray around~$v$ is in $(0, \pi)$, that ray is selected. Otherwise, no ray is selected.
	\item Case 2: If the wedge around~$v$ contains some ray internally, 
		then the wedge is removed from consideration and a ray is chosen as in Case~1.
	\item Case 3: If the wedge around~$v$ contains no ray internally, then the wedge is replaced with its 
		counterclockwisemost representative ray, and then a ray is selected as in Case~1.
\end{compactitem}
\begin{figure}[h!]
\centering
\scalebox{0.5}{
\begin{tikzpicture}[every path/.style={>=latex}]
\filldraw[fill=black] (0, 0) circle (1mm);
\node (A) at (0, 0){};
\node (B) at (3, 1){};
\node (C) at (3.2, 0){};
\node (F) at (0.5, 3){};
\node (D) at (2.7, 1.6){};
\node (E) at (-3, 1){};
\path[->, line width=2]
	(A) edge (B)
    (A) edge (F)
    (A) edge[color=red, line width=3] (E);

\end{tikzpicture}
\begin{tikzpicture}[every path/.style={>=latex}]
\filldraw[fill=black] (0, 0) circle (1mm);
\filldraw[fill=black!20!white, draw=black!50!white]
    (0,0) -- (30mm,0mm) arc (0:31:30mm) -- (0,0);
\node (A) at (0, 0){};
\node (B) at (3, -1){};
\node (C) at (3.2, 0){};
\node (F) at (0.5, -3){};
\node (D) at (2.7, 1.6){};
\node (E) at (3, 1){};
\path[->, line width=2]
	(A) edge (B)
    (A) edge (F)
    (A) edge[color=red, line width=3] (E);
\path[->, line width=1]
    (A) edge (C)
    (A) edge (D);
\end{tikzpicture}
\begin{tikzpicture}[every path/.style={>=latex}]
\filldraw[fill=black] (0, 0) circle (1mm);
\filldraw[fill=red!20!white, draw=red!50!white]
    (0,0) -- (30mm,0mm) arc (0:31:30mm) -- (0,0);
\node (A) at (0, 0){};
\node (B) at (3, -1){};
\node (C) at (3.2, 0){};
\node (F) at (0.5, -3){};
\node (D) at (2.7, 1.6){};
\node (E) at (-0.5, -3){};
\path[->, line width=2]
	(A) edge (B)
    (A) edge (F)
    (A) edge (E);
\path[->, line width=1]
    (A) edge[color=red, line width=1] (C)
    (A) edge[color=red, line width=3] (D);
\end{tikzpicture}
\begin{tikzpicture}[every path/.style={>=latex}]
\filldraw[fill=black] (0, 0) circle (1mm);

\node (A) at (0, 0){};
\node (B) at (3, 1){};
\node (C) at (3.2, 0){};
\node (F) at (0.5, -3){};
\node (D) at (2.7, 1.6){};
\node (E) at (-3, 1){};
\path[->, line width=2]
	(A) edge (B)
    (A) edge (F)
    (A) edge (E);
\end{tikzpicture}
\begin{tikzpicture}[every path/.style={>=latex}]
\filldraw[fill=black] (0, 0) circle (1mm);
\filldraw[fill=black!20!white, draw=black!50!white]
    (0,0) -- (30mm,0mm) arc (0:31:30mm) -- (0,0);
\node (A) at (0, 0){};
\node (B) at (3, 1){};
\node (C) at (3.2, 0){};
\node (F) at (0.5, -3){};
\node (D) at (2.7, 1.6){};
\node (E) at (-3, 1){};
\path[->, line width=2]
    (A) edge (F)
    (A) edge (E);
\path[->, line width=1]
    (A) edge (C)
    (A) edge (D);
\end{tikzpicture}}

\caption{Examples of the described selection about a vertex; in the last two cases, no ray or wedge is chosen.}
\end{figure}
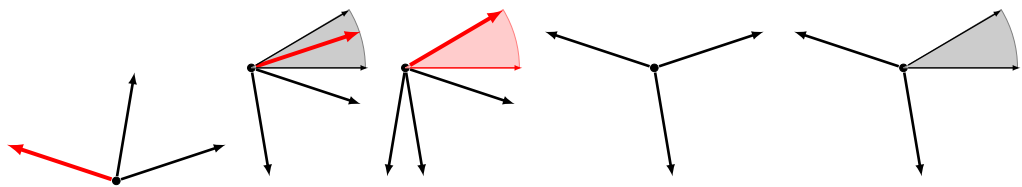

Every convex set is chosen by one of its vertices. First, we observe that this holds for all edges. 
Indeed, suppose this is not the case for some edge~$\overline{v_i v_j}$. Since ray $\overrightarrow{v_i v_j}$ 
was not chosen, some convex set containing~$v_i$ as a vertex lies entirely in union of the open 
half-plane~$H^+$ with the extension of $\overrightarrow{v_j v_i}$ past~$v_i$ (see Figure ~\ref{fig:EdgeChoose}). Similarly, 
some convex set containing~$v_j$ as a vertex must lie entirely in the intersection of the open 
half-plane~$H^-$ with the extension of $\overrightarrow{v_i v_j}$ past~$v_j$. However, 
these two sets are disjoint, so the corresponding convex sets would also be disjoint. 
This contradicts the condition $|C_i\cap C_j\cap W|=1$ for all $i, j$, and is therefore impossible. 
It follows that every two-vertex convex set is chosen by one of its vertices.

\begin{figure}[h!]
\centering
\minipage{0.5\textwidth}
\scalebox{0.6}{
\begin{tikzpicture}[every path/.style={>=latex}]
\tikzstyle{every node}=[font=\huge]
\fill[fill=blue!20!white](-7, 0)rectangle (7, 4);
\fill[fill=red!20!white](-7, 0)rectangle (7, -4);
\filldraw[fill=black] (-3, 0) circle (1mm);
\filldraw[fill=black] (3, 0) circle (1mm);
\node (A) at (-3, 0){};
\node (B) at (0, 2){$H^+$};
\node (C) at (0, -2){$H^-$};
\node (F) at (3, 0){};
\node (M) at (3, 0.5){$v_j$};
\node (N) at (-3, 0.5){$v_i$};
\node (D) at (7, 0){};
\node (E) at (-7, 0){};
\path[-, line width=3]
	(A) edge (F);
\path[->, line width=3]
    (A) edge[color=red] (E)
    (F) edge[color=blue] (D); 
\end{tikzpicture}}
\caption{Illustration of how edges are chosen}
\label{fig:EdgeChoose}
\endminipage\hfill
\minipage{0.5\textwidth}
\centering
\scalebox{.75}{
\begin{tikzpicture}
\filldraw[fill=black!20!white, draw=black, line width=2](2, 0)--(.62, 1.9)--(-1.61, 1.17)--(-1.61,-1.17)--(.62,-1.9)--(2,0);
\node (v1) at (2.2, 0){$v_1$};
\node (v2) at (.62, 2.1){$v_2$};
\node (v3) at (-1.71, 1.27){$v_3$};
\node (v4) at (-1.71,-1.27){$v_4$};
\node (v5) at (.62,-2.1){$v_5$};
\node (C) at (0,0){$C$};
\node (A1) at (2,2.54){$A_1$};
\node (A2) at (-1.81,2.63){$A_2$};
\node (A3) at (-2.99,-.73){$A_3$};
\node (A4) at (-.23, -3.07){$A_4$};
\node (A5) at (2.85,-1.17){$A_5$};

\path[dashed,-, line width=2]
	(2, 0) edge (3.38, 1.9)
    (.62,1.9) edge (-.76,3.8)
    (-1.61,1.17) edge (-3.84, .44)
    (-1.61,-1.17) edge (-1.61,-3.51)
    (.62,-1.9) edge (2.85,-2.63);
\end{tikzpicture}}
\caption{Illustration of how two-dimensional sets are chosen.}
\endminipage\hfill
\end{figure}

Now it suffices to check the statement for nonedge convex sets. Suppose for the sake 
of contradiction that a convex set $C=\conv\{v_1,v_2,\ldots v_k\}$ has the 
property that no $v_i$, $1\le i\le k$, chose the wedge corresponding to~$C$. 
Let $v_1, \dots, v_k$ be ordered in counterclockwise order around the boundary of~$C$. 
For each $i, 1\le i\le k$, let $R_i$ denote the ray $\overrightarrow{v_{i-1}v_i}$, and let $A_i$ 
denote the closed wedge between rays $R_i$ and $R_{i+1}$ with indices taken modulo~$k$. 
The convex set $C$, along with $A_1, \dots, A_k$ then form a partition of the plane.

Since all pairs of nonedge convex sets are assumed to be vertex disjoint, the only other 
sets that could possibly contain the $v_i, 1\le i\le k,$ as vertices are edges. If, for any~$i, 1\le i\le k$, 
all rays from $v_i$ (if there are any) point towards the interior of the region~$A_i$, then the vertex $v_i$ would choose 
the wedge corresponding to~$C$. Furthermore, no ray can point alongside an edge of~$C$, as the 
intersection of that edge with $C$ would necessarily contain two vertices. Therefore, we may 
assume that, for every $i, 1\le i\le k$, some ray either points inside the wedge corresponding to~$C$, 
or points into the union of $R_i$ and the unique open half-plane $H_i$ disjoint from $C$ and whose defining 
line is $\overleftrightarrow{v_{i-1}, v_i}$. There are two cases:

\textbf{Case 1:} For every $i$, some ray at $v_i$ points inside~$C$. 
\\ Observe that no two such rays may meet inside $C$. Otherwise, the intersection of either 
corresponding segment with $C$ would necessarily contain two points in~$W$. It follows that 
the edges corresponding to these rays meet outside of $C$, so that every edge must intersect 
the boundary of $C$ internally. Let any segment from $v_1$ which points inside $C$ intersect 
the boundary of $C$ again at a point~$Y$. 

Since the intersection of this segment and $C$ already contains $v_1\in W$, it follows that $Y$ is not a vertex of $C$, so that it lies on some edge. Let $v_i, i\neq 1$ be one of the vertices of the edge of $C$ containing $Y$. Then any edge with a vertex at $v_i$ must, in order to intersect $\overrightarrow{v_iY}$ outside of $C$, also point outside of~$C$. This contradicts the assumption that every vertex has some ray pointing inwards, so this case is resolved. 

\textbf{Case 2:} For some $i$, there is a ray from $v_i$ which points into~$H_i$. 
\\In this case, no ray from $v_{i-1}$ can point inside $C$, for then this ray and the above ray from $v_i$ would point into opposite sides of the line $v_{i-1}v_i$. It follows that some ray from $v_{i-1}$ points into $H_{i-1}$. Repeating this argument $k-2$ more times, there is some ray $r_i$ for each $i, 1\le i\le k$ which points into $H_i$. Let $\theta_i$ denote the clockwise angle measured between rays $r_i$ and $\overrightarrow {v_iv_{i-1}}$, and let $\gamma_i$ denote the measure of $\angle v_{i+1}v_iv_{i-1}$.

For each $i$, the rays $r_i$ and $r_{i+1}$ must intersect, since the corresponding segments intersect. The condition that $r_i, r_{i+1}$ intersect is exactly the condition that the sum of the clockwise angle measures from $\overline{v_iv_{i+1}}$ to $r_i$ and from $r_{i+1}$ to $\overline{v_{i+1}v_i}$ is less than $\pi$; that is, $(2\pi-\gamma_i-\theta_i)+\theta_{i+1}<\pi$, or $\theta_{i+1}-\theta_i<\gamma_i-\pi$ for each $1\le i\le k$. However, summing these $k$ inequalities cyclically gives:
$$
0=\sum_{i=1}^k(\theta_{i+1}-\theta_i)
 <\sum_{i=1}^k(\gamma_i-\pi)
 =-2\pi
$$
This is a contradiction, so this case is also impossible.

Since a contradiction was derived in all cases, it follows that some vertex from every convex set does in fact choose that convex set. Since each vertex chooses at most one convex set, this mapping forms a natural surjection from a subset of vertices onto $\{C_1, \dots, C_m\}$. It follows that $m\le n$ as required.
\end{proof}

The following theorem shows that Conjecture~\ref{conj:EdgesBoundedByVertices} holds whenever the transversal set~$W$
contains only the vertices and no additional points as in Example~\ref{ex:proj}. This is a purely combinatorial statement 
independent of any geometry of the sets~$C_i$ and ambient space.

\begin{thm} 
\label{thm:comb-thrackle}
	Let $C_1, \dots, C_m$ be sets and suppose there exists a transversal of their pairwise intersections~$W$, 
	that is $|C_i\cap C_j \cap W|=1$ for all $i\ne j$. Then $m \le |W|$.
\end{thm}

\begin{proof}
Create a graph where the vertices represent the sets $C_i$ and there is an edge between the vertices if the two corresponding sets intersect. 
Since every pair of sets must intersect, this graph will be the complete graph on $m$ vertices,~$K_m$.
Any point of the transversal set~$W$ induces a complete subgraph of sets it intersects. Therefore $W$ induces 
a decomposition of the complete graph into proper complete subgraphs. The complete graph $K_m$ cannot be 
decomposed into less than $m$ proper complete subgraphs; see de Brujin and Erd\H{o}s~\cite{erdos1948}. Thus~${m\le |W|}$.
\end{proof}

While this is a purely combinatorial statement, Conjecture~\ref{conj:EdgesBoundedByVertices} has geometric content and
the analogous statement fails in~$\R^3$:

\begin{figure}[h!]
\centering

\begin{tikzpicture}
 \tikzstyle{every node}=[circle, draw, fill=black!50,
                       inner sep=0pt, minimum width=4pt]
                       
\filldraw[fill=black!20!white, draw=black, line width=2](0, -3)--(-1, 1.5)--(3, 0)--(0,-3);

 \path[black,-, line width=2]
    (-1,1.5) edge (1,-1.5)
    (0,3) edge (0,-3)
    (3,0) edge (-3,0);

\filldraw[fill=black!20!white, draw=black, line width=2](0, -3)--(1, -1.5)--(-3, 0)--(0,-3);

 \filldraw[fill=black!20!white, draw=black, line width=2](0, 3)--(3, 0)--(1, -1.5)--(0,3);
\filldraw[fill=black!20!white, draw=black, line width=2](0, 3)--(-1, 1.5)--(-3, 0)--(0,3);

\end{tikzpicture}

\caption{Counterexample to Conjecture~\ref{conj:EdgesBoundedByVertices} in~$\R^3$ on six vertices with seven convex sets.}

\end{figure}
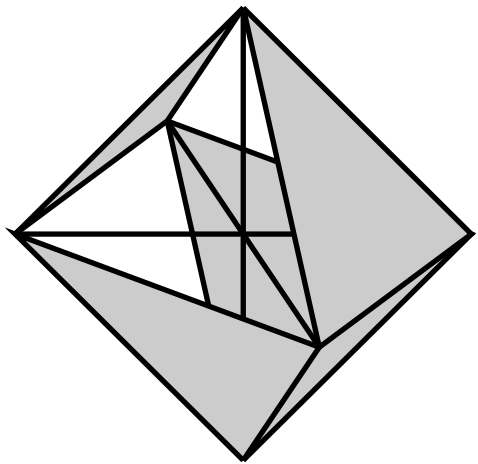

\section{Higher-dimensional thrackles}
\label{sec:high-dim}

\noindent
A $d$-dimensional simplicial complex is \emph{pure} if every face is contained in a $d$-dimensional face.
A pure simplicial complex $K$ of dimension~$d$ is called \emph{$d$-thrackle} if there is a continuous map
${f \colon K \longrightarrow \R^{d+1}}$ such that
\begin{compactenum}[(i)]
	\item the restriction of~$f$ to any facet is an embedding,
	\item any two facets intersect in a $(d-1)$-ball,
	\item intersections between faces are \emph{stable}, that is, there is an $\varepsilon > 0$ such that any 
		homotopy that moves points by at most $\varepsilon$ cannot remove the intersection.
\end{compactenum}

The $(d-1)$-faces of a $d$-thrackle are called \emph{ridges}.
If the map $f$ is linear on each facet then we call $K$ \emph{linear $d$-thrackle}.
The classical case of thrackle graphs corresponds to $1$-thrackles.
Here we prove higher-dimensional extensions of the linear thrackle conjecture:

\begin{thm}
\label{thm:thrackle-highdim}
	A linear $(d-1)$-thrackle with $m$ facets and $n$ ridges satisfies $dm\le 2n$.
\end{thm}

\begin{proof}
For any pure $(d-1)$-dimensional simplicial complex with $m$ facets and $n$ ridges such that any ridge
is contained in at most two facets we have that $dm \le 2n$ by multiple counting. Suppose there is a 
$(d-1)$-thrackle $K$ with $m$ facets and $n$ ridges such that $dm > 2n$. Further suppose that $K$ is
a minimal counterexample, that is, any $(d-1)$-thrackle with at most $m-1$ facets satisfies the inequality
of the theorem.

The simplicial complex $K$ contains a ridge $\tau$ that is contained in at least three facets. Let $\sigma_1,
\sigma_2, \sigma_3$ be three facets incident to~$\tau$. Fix any affine map ${f \colon K \longrightarrow \R^d}$
that realizes $K$ as a linear $(d-1)$-thrackle. Since $f$ embeds each facet, the $(d-1)$-simplices $f(\sigma_i)$
span affine hyperplanes~$H_i$. These hyperplanes intersect in the $(d-2)$-plane spanned by~$f(\tau)$,
and at most two of the hyperplanes can coincide. Thus at least one of the hyperplanes $H_j$ leaves the~$f(\sigma_i)$,
$i \ne j$, on different sides of it, meaning that $f(\sigma_i) \setminus f(\tau)$, $i \ne j$, are contained in different 
open halfspaces determined by~$H_j$. We claim that $\sigma_j$ is only adjacent to other facets through~$\tau$
and not through any other ridge. This is because any facet~$\sigma$ that shares a ridge with~$\sigma_j$
has its image $f(\sigma)$ entirely contained in one closed halfspace determined by~$H_j$. But unless $\sigma$
contains $\tau$ the $(d-1)$-simplex $f(\sigma)$ cannot intersect both $f(\sigma_i)$, $i \ne j$, in $(d-2)$-balls.

Removing $\sigma_j$ yields a $(d-1)$-thrackle with $m-1$ facets and $n-d+1$ ridges. Now $d(m-1) =dm-d > 2n-d
\ge 2(n-d+1)$ and thus we obtained a counterexample with fewer facets than~$K$, in contradiction to the 
minimality of~$K$.
\end{proof}

Any embedding of the boundary of the $d$-simplex into $\R^d$ is a $(d-1)$-thrackle with $d+1$ facets and $\binom{d+1}{2}$ ridges. 
Thus the bound in Theorem~\ref{thm:thrackle-highdim} is tight in any dimension.
The proof shows that the only examples of $(d-1)$-thrackles, $d \ge 3$, with equality $dm = 2n$ are pseudomanifolds
in the sense that each ridge is contained in precisely two facets.

If in the definition of $d$-thrackle we only require that any two facets intersect in a contractible set instead of
a $(d-1)$-ball, Theorem~\ref{thm:thrackle-highdim} fails to hold in this more general setting: consider a square pyramid
with base $1,2,3,4$ in cyclic order and apex~$5$. Let the set of facets consist of all triangles of the pyramid in addition to
the triangle $1,2,3$ and its three cyclic copies as well as the triangles $1,3,5$ and $2,4,5$. Every pair of facets intersects
in a ball of dimension at most two. There are ten facets and ten ridges, which violates the inequality of 
Theorem~\ref{thm:thrackle-highdim}.

Moreover, for a $(d-1)$-thrackle with $m$ facets, the bound of $m\le |V|$ will not hold in $\R^d$ as can be seen by the counterexample
in the figure below. In this figure, all edges will be extended into triangles to the blue vertex directly above 
the star, and the three marked edges will be extended to triangles with the red vertex above and to the side of the star.

\begin{figure}[h!]
\centering
\begin{tikzpicture}
\label{fig:counterexample-3d}
 \tikzstyle{every node}=[circle, draw, fill=black!50,
                       inner sep=0pt, minimum width=4pt]
                       
        \filldraw[fill=blue] (0, 0) circle (1.2mm);
 
        \node (A) at (360/7:3){};   
        \node (B) at (720/7:3){};
        \node (C) at (1080/7:3){}; 
        \node (D) at (1440/7:3){}; 
        \node (E) at (1800/7:3){}; 
        \node (F) at (2160/7:3){}; 
        \node (G) at (0:3){};
        \node (H) at (2200/7:6){};  
        \node (I) at (2120/7:6){};  
        
        \filldraw[fill=red] (2160/7:4.5) circle (1.2mm);
        
        \path[black,-, line width=2]
    (A) edge (D)
    (D) edge (G)
    (G) edge (C)
    (C) edge (F)
    (F) edge (B)
    (B) edge (E)
    (E) edge (A);
    
    \path[dashed,-, line width=2]
    (F) edge (H)
    (F) edge (I);
    
    \path[red,-, line width=2]
    (A) edge (D)
    (G) edge (C)
    (B) edge (E);
        
\end{tikzpicture}
\caption{Coning all edges to the blue vertex and red edges to the red vertex yields a 
linear $2$-thrackle in~$\R^3$ with ten facets and nine vertices.}
\end{figure}
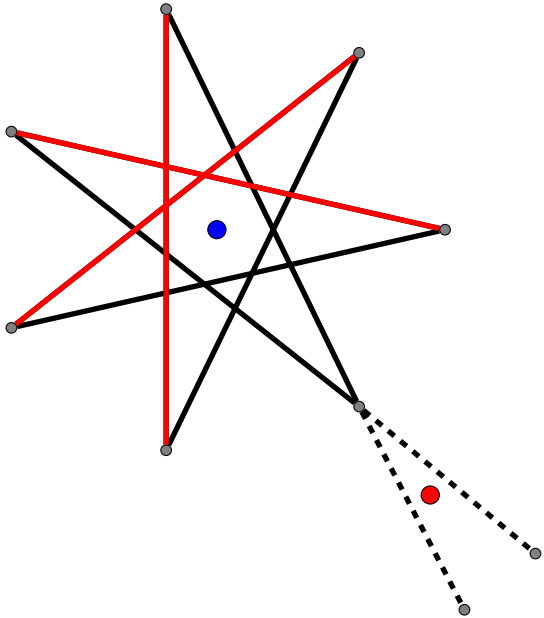

We conjecture the continuous analog of Theorem~\ref{thm:thrackle-highdim}.

\begin{conj}
\label{conj:d-thrackle}
	Let $K$ be a $(d-1)$-thrackle with $m$ facets and $n$ ridges. Then $dm \le 2n$.
\end{conj}

The planar case $d=2$ of Conjecture~\ref{conj:d-thrackle} is Conway's thrackle conjecture. 
Under mild assumptions on the map~$f$ we can show that Conway's thrackle conjecture
in fact implies Conjecture~\ref{conj:d-thrackle}: suppose $f\colon K \longrightarrow \R^{d+1}$
realizes $K$ as a $d$-thrackle in such a way that for every vertex~$v$ of~$K$ we can find
a $d$-sphere $S_v$ around~$f(v)$ that intersect all facets incident to~$v$ in $(d-1)$-balls and
for every pair of distinct facets $\sigma$ and $\tau$ incident to~$v$ the intersection
$f(\sigma) \cap f(\tau) \cap S_v$ is a $(d-2)$-ball and stable within~$S_v$. That is, $S_v$ is 
a sphere that is in general position with respect to the image of~$f$ restricted to the star
of~$v$. Then stereographic projection realizes the link of~$v$ as a $(d-1)$-thrackle.

These mild assumptions on $f$ are met, for example, when $f$ embeds each facet $\sigma$ of~$K$ into~$\R^{d+1}$ as a $d$-manifold with boundary (i.e. $f$ can be smoothly extended past~$\partial\sigma$), and stability is strengthened to the condition of these manifolds intersecting transversally. If a vertex $v$ is contained in a face $\sigma$, then $f(v)$ is on the boundary of $f(\sigma)$, so for sufficiently small~$\varepsilon$, the open ball $B_\varepsilon$ of radius $\varepsilon$ centered at $f(v)$ has $B_\varepsilon\cap f(\sigma)$ homeomorphic to a closed half plane in~$\R^{d}$, or equivalently an open $d$-ball with an open $(d-1)$-ball pasted on the boundary. If $S_\varepsilon$ is the sphere of radius $\varepsilon$ centered at~$f(v)$, we have $S_\varepsilon \cap f(\sigma) = (\overline{B_\varepsilon} \cap f(\sigma))-(B_\varepsilon \cap f(\sigma))$. The right-hand side is homeomorphic (via the above homeomorphism) to a closed $d$-ball minus its interior and minus an open $(d-1)$-ball on its boundary, which is homeomorphic to precisely a closed $(d-1)$-ball. Note that for faces $\sigma, \tau$ that contain the vertex~$v$, $f(v)$ is also on the boundary of the $(d-1)$-manifold with boundary $f(\sigma)\cap f(\tau)$, so we can use the same argument to find $\varepsilon$ small so that $S_\varepsilon$ also intersects each $f(\sigma)\cap f(\tau)$ in a $(d-2)$-ball. Furthermore, the condition that $f(\sigma)\cap S_\varepsilon$ and $f(\tau)\cap S_\varepsilon$ intersect transversally on the sphere is equivalent to the condition that $S_\varepsilon$ and $f(\sigma)\cap f(\tau)$ intersect transversally, which is possible for some small perturbation of the sphere since manifold transversality is known to be a generic property; see Chapter 3 of Hirsch~\cite{hirsch2012} for an introduction to transversality.

These assumptions on $f$ allow us to inductively transfer inequalities relating edges and vertices of thrackles in $\R^2$ to inequalities relating facets and ridges of $d$-thrackles in $\R^{d+1}$. Supposing we have the inequality $dm \le 2cn$ between the number of facets~$m$
and the number of ridges~$n$ of a $(d-1)$-thrackle for some constant~$c$, we get the
inequality $(d+1)m \le 2cn$ for $d$-thrackles in $\R^{d+1}$ realized as above by multiple counting: 
let $m$ be the number of facets of~$K$ and $n$ the number of ridges. Denote by $f_{k}(v)$ 
the number of $k$-faces in the link of~$v$, that is $f_{d-1}(v)$ is the number of facets incident 
to~$v$. We have the inequality $df_{d-1}(v) \le 2cf_{d-2}(v)$ for every vertex link.
Summing this inequality over all vertex links yields $(d+1)dm \le 2cdn$.

Thus, when $f$ satisfies the above assumptions, the bound for plane thrackles given by Fulek and Pach~\cite{fulek2010} shows that any $(d-1)$-thrackle with $m$ facets
and $n$ ridges satisfies the inequality $dm \le 2.856n$. A proof of the thrackle conjecture immediately
implies our Conjecture~\ref{conj:d-thrackle} for such $f$ as noted above. High-dimensional versions of this conjecture might be 
simpler to attack since there are more serious restrictions on $(d-1)$-thrackles for $d \ge 3$: for example,
every vertex link has to be a $(d-2)$-thrackle.

\bibliographystyle{amsplain}

\begin{thebibliography}{10}

\bibitem{Bacher2011}
Roland Bacher, G{\"u}nter~M. Ziegler, and {others}, \emph{An
  {E}rd{\H{o}}s--{S}zekeres-type question}, June 2011, See
  \url{http://mathoverflow.net/questions/67762/}.

\bibitem{barany2016}
Imre B{\'a}r{\'a}ny, Pavle V.~M. Blagojevi{\'c}, and G{\"u}nter~M. Ziegler,
  \emph{{Tverberg's Theorem at 50: Extensions and Counterexamples}}, Notices
  Amer. Math. Soc. \textbf{63} (2016), no.~7, 732--739.

\bibitem{barany1992}
Imre B{\'a}r{\'a}ny and David~G. Larman, \emph{{A colored version of Tverberg's
  theorem}}, J. London Math. Soc. \textbf{2} (1992), no.~2, 314--320.

\bibitem{blagojevic2015-2}
Pavle V.~M. Blagojevi{\'c}, Florian Frick, and G{\"u}nter~M. Ziegler,
  \emph{{Barycenters of Polytope Skeleta and Counterexamples to the Topological
  Tverberg Conjecture, via Constraints}}, arXiv preprint arXiv:1510.07984
  (2015).

\bibitem{blagojevic2015}
Pavle V.~M. Blagojevi{\'c}, Benjamin Matschke, and G{\"u}nter~M. Ziegler,
  \emph{{Optimal bounds for the colored Tverberg problem}}, J. Eur. Math. Soc.
  \textbf{17} (2015), no.~4, 739--754.

\bibitem{delongueville2001}
Mark de~Longueville, \emph{{Notes on the topological Tverberg theorem}},
  Discrete Math. \textbf{241} (2001), no.~1, 207--233.

\bibitem{doignon1977}
Jean-Paul Doignon and G.~Valette, \emph{Radon partitions and a new notion of
  independence in affine and projective spaces}, Mathematika \textbf{24}
  (1977), no.~01, 86--96.

\bibitem{erdos1946}
Paul Erd{\H o}s, \emph{{On sets of distances of $n$ points}}, Amer. Math.
  Monthly \textbf{53} (1946), no.~5, 248--250.

\bibitem{erdos1948}
Paul Erd{\H o}s and Nicolaas~G. de~Bruijn, \emph{On a combinatioral [sic]
  problem}, Indagationes Mathematicae \textbf{10} (1948), 421--423.

\bibitem{frick2015}
Florian Frick, \emph{{Counterexamples to the topological Tverberg conjecture}},
  Oberwolfach Reports \textbf{12} (2015), no.~1, 318--321.

\bibitem{fulek2010}
Radoslav Fulek and J{\'a}nos Pach, \emph{{A computational approach to
  Conway's thrackle conjecture}}, Comput. Geom. \textbf{44} (2011),
  345--355.

\bibitem{hirsch2012}
Morris~W. Hirsch, \emph{Differential topology}, vol.~33, Springer Science \&
  Business Media, 2012.

\bibitem{lovasz1997}
L{\'a}szl{\'o} Lov{\'a}sz, J{\'a}nos Pach, and Mario Szegedy, \emph{{On
  Conway's thrackle conjecture}}, Discrete Comput. Geom. \textbf{18} (1997),
  no.~4, 369--376.

\bibitem{mabillard2015}
Isaac Mabillard and Uli Wagner, \emph{{Eliminating Higher-Multiplicity
  Intersections, I. A Whitney Trick for Tverberg-Type Problems}}, arXiv
  preprint arXiv:1508.02349 (2015).

\bibitem{perles2007}
Micha~A. Perles and Moriah Sigron, \emph{{A generalization of Tverberg's
  theorem}}, arXiv preprint arXiv:0710.4668 (2007).

\bibitem{perles2014}
Micha~A. Perles and Moriah Sigron, \emph{Strong general position}, arXiv preprint arXiv:1409.2899 (2014).

\bibitem{reay1979}
John~R. Reay, \emph{{Several generalizations of Tverberg's theorem}}, Israel
  J. Math. \textbf{34} (1979), no.~3, 238--244.

\bibitem{reay1979-2}
John~R. Reay, \emph{Twelve general position points always form three intersecting
  tetrahedra}, Discrete Math. \textbf{28} (1979), no.~2, 193--199.

\bibitem{tverberg1966}
Helge Tverberg, \emph{{A generalization of Radon's theorem}}, J. London Math.
  Soc. \textbf{41} (1966), no.~1, 123--128.

\bibitem{vanlint2001}
Jacobus~H. van Lint and Richard~M. Wilson, \emph{A course in combinatorics},
  Cambridge University Press, 2001.

\end{thebibliography}

%%%%%%%%%%%%%%%%%%%%%%%%%%%%%%%%%%%%%%%%%%%%%%%%%%%%%%%%%%%%%%%%%%%%%%%%%%%%%%%%%%%%%

\end{document}